\documentclass{article}

\usepackage{amssymb}
\usepackage{amsfonts}
\usepackage{amsmath}

\newtheorem{thm}{Theorem}[section]
\newtheorem{ack}{Acknowledgement}[section]
\newtheorem{con}{Conclusion}[section]

\newtheorem{cor}{Corollary}[section]
\newtheorem{defi}{Definition}[section]

\newtheorem{prop}{Proposition}[section]

\newenvironment{proof}[1][Proof]{\noindent\textbf{#1.} }{\ \rule{0.5em}{0.5em}}
\numberwithin{equation}{section}
\def\({\left ( }
\def\){\right )}
\def\<{\left < }
\def\>{\right >}

\begin{document}

\title{\textbf{The Bi-periodic Fibonacci Octonions}}
\author{Nazmiye Yilmaz\thanks{e mail: \ \textit{nzyilmaz@selcuk.edu.tr, yyazlik@nevsehir.edu.tr,
ntaskara@selcuk.edu.tr}} \\
Department of Mathematics, Science Faculty, \\
Selcuk University, Campus, Konya, Turkey \and Yasin Yazlik \\
Department of Mathematics, Faculty of Science and Letters, \\
Nevsehir Haci Bektas Veli University, Nevsehir, Turkey \and Necati Taskara \\
Department of Mathematics, Science Faculty, \\
Selcuk University, Campus, Konya, Turkey }

\maketitle

\begin{abstract}
In this paper, by using bi-periodic Fibonacci numbers, we introduce the bi-periodic Fibonacci octonions. After that, we derive the generating function of these octonions as well as investigated some properties over them. Also, as another main result of this paper, we give the summations for bi-periodic Fibonacci octonions.
\par
\textbf{Ams Classification:} 11B39, 11R52, 15A66, 20G20.
\end{abstract}

\section{Introduction}
\qquad
There has been an increasing interest on quaternions and octonions that play an important role in various areas such as computer sciences, physics, differential geometry, quantum physics, signal, color image processing, geostatics and analysis \cite{Adler,Dixon}.
\par
Let $\mathcal{O}$ be the octonion algebra over the real number field $\mathbb{R}$. It is known, by the Cayley-Dickson process that any $p\in\mathcal{O}$ can be written as
\begin{equation*}
p=p{'}+p{''}e \,, 
\end{equation*}
where \small{${p}',{p}''\in H=\left\{ {{a}_{0}}+{{a}_{1}}i+{{a}_{2}}j+{{a}_{3}}k\left| \,{{i}^{2}}={{j}^{2}}={{k}^{2}}=-1,ijk=-1,\,{{a}_{0}},{{a}_{1}},{{a}_{2}},{{a}_{3}}\in \mathbb{R} \right. \right\}$} the real quaternion division algebra. The octonions in Clifford algebra $\boldmath{C}$ are a normed division algebra with eight dimensions over the real numbers larger than the quaternions. The field $\mathcal{O}\cong\boldmath{C}^4$ of octonions 
\begin{equation*}
p=\sum_{s=0}^{7}p_{s}e_{s},\ \ p_{s}\in\mathbb{R},
\end{equation*}
is an eight-dimensional non-commutative and non-associative $\mathbb{R}$-field generated by eight base elements
\begin{equation*}
e_{0}=1,\, e_{1}=i,\, e_{2}=j, \,e_{3}=k, \,e_{4}=e,\, e_{5}=ie,\, e_{6}=je,\, e_{7}=ke\,.
\end{equation*}

The addition and multiplication of any two octonions, $p={p}'+{p}''e$, $q={q}'+{q}''e$, are defined by
\begin{equation*}
p+q=p{'}+q{'}+(p{''}+q{''})e \ \text{    and    } \  pq=p{'}q{'}-\overline{q{''}}p{''}+(q{''}p{'}+p{''}\overline{q{'}})e\,,
\end{equation*}
where $\overline{q{'}}, \overline{q{''}}$ denote the conjugates of the quaternions $q{'}, q{''}$, respectively.
The conjugate and norm of $p$ are denoted by
\begin{equation}
\overline{p}=Re(p)-Im(p) \ \text{and} \ p\overline{p}=\sum_{s=0}^{7}p_{s}^2 \,,
\label{eslenik}
\end{equation}
respectively \cite{Adler,Dixon,Tian,Ward}.
\par
On the other hand, the literature includes many papers dealing with the special number sequences such as Fibonacci, Lucas, Pell (\cite{Bilgici,EdsonYayenie,IrmakAlp,Koshy,PanarioSahinWang,Sahin,Yayenie,YazlikTaskara,YilmazYazlikTaskara}). One of these directions goes through to the \textit{bi-perodic Fibonacci} (or, equivalently, \textit{generalized Fibonacci}) and the \textit{bi-periodic Lucas} (or, equivalently, \textit{generalized Lucas}). In fact bi-perodic Fibonacci sequences have been firstly defined by Edson and Yayenie in 2009 and then some important properties for these sequences have been investigated in the references \cite{EdsonYayenie,Yayenie}. Also, bi-perodic Lucas sequences have been defined by Bilgici in 2014 and then some important properties for these sequences have been investigated in the reference \cite{Bilgici}. For any two nonzero real numbers $a$ and $b$, the bi-periodic Fibonacci $\left\{q_{n}\right\} _{n=0}^{\infty }$ and bi-periodic Lucas sequences $\left\{l_{n}\right\} _{n=0}^{\infty }$ are defined recursively by
\begin{equation}
q_{0}=0,\ q_{1}=1,\ q_{n}=\left\{
\begin{array}{c}
aq_{n-1}+q_{n-2},\ \ \text{if }n\text{ is even} \\
bq_{n-1}+q_{n-2},\ \ \text{if }n\text{ is odd}\
\end{array}
\right. \ \ n\geq 2 \,  \label{rekFib}
\end{equation}
and
\begin{equation}
l_{0}=2,\ l_{1}=a,\ l_{n}=\left\{
\begin{array}{c}
bl_{n-1}+l_{n-2},\ \ \text{if }n\text{ is even} \\
al_{n-1}+l_{n-2},\ \ \text{if }n\text{ is odd}\
\end{array}
\right. \ \ n\geq 2 \,.  \label{rekLuc}
\end{equation}
\par
In addition, the authors in the references \cite{Bilgici},\cite{EdsonYayenie} and \cite{Yayenie} expressed so many properties on the bi-periodic Fibonacci and bi-periodic Lucas sequences. In fact some of the main outcomes (that depicted in these references) of these sequences can be summarized as in the following:   

\begin{itemize}
\item The Binet formulas are given by
\begin{equation}
q_{n}=\dfrac{a^{1-\xi \left( n\right) }}{\left(ab\right) ^{\left\lfloor \frac{n}{2}\right\rfloor }}\left( \dfrac{\alpha
^{n}-\beta ^{n}}{\alpha -\beta }\right)  \label{binFib}
\end{equation}
and
\begin{equation}
l_{n}=\dfrac{a^{\xi \left( n\right) }}{\left(ab\right) ^{\left\lfloor \frac{n+1}{2}\right\rfloor }}\left(\alpha
^{n}+\beta ^{n}\right),\label{binLuc}
\end{equation}
 where the condition $\xi \left( n\right)
=n-2\left\lfloor \frac{n}{2}\right\rfloor $ can be read as
\begin{equation}
\xi \left( n\right) =\left\{
\begin{array}{c}
0, \ \ n\text{ is even} \\
1, \ \ n\text{ is odd}\
\end{array}
\right.\label{ksi}
\end{equation}
and $\alpha,\beta$ are roots of characteristic equation of $\lambda^2-ab\lambda-ab=0$.
\item The generating functions for the bi-periodic Fibonacci and bi-periodic Lucas sequences with odd subscripted are
\begin{equation}
f(x)=\sum_{m=1}^{\infty}q_{2m-1}x^{2m-1}\dfrac{x-x^3}{1-(ab+2)x^2+x^4}\,,\label{ureFib}
\end{equation}
and
\begin{equation}
L_{1}(x)=\sum_{m=0}^{\infty}l_{2m+1}x^{2m+1}\dfrac{a+ax^3}{1-(ab+2)x^2+x^4}\,.\label{ureLuc}
\end{equation}
\item The bi-perodic Fibonacci and bi-peridoic Lucas sequences satify the equations
\begin{equation}
(ab+4)q_{n}=l_{n-1}+l_{n+1}\,,\label{bag}
\end{equation}
\begin{equation}
l_{n}=q_{n-1}+q_{n+1}\,\label{bag1}
\end{equation}
and
\begin{equation}
q_{-n}=(-1)^{n-1}q_{n}\,,\label{negFib}
\end{equation}
\begin{equation}
l_{-n}=(-1)^{n}l_{n}\,.\label{negLuc}
\end{equation}
\end{itemize}

\par
After all these above material on bi-periodicity, let us give our attention to the other classification parameters on these above special sequences, namely quaternions and octonions. We should note that, in the literature (see \cite{AkkusKecilioglu,AkyigitKosalTosun,Catarino}-\cite{CimenIpek}, \cite{FlautSavin}-\cite{IpekArı},\cite{Iyer1}-\cite{KeciliogluAkkus},\cite{Polatli}-\cite{Ramirez}, \cite{Savin}-\cite{TasciYalcin} and the references cited in them), it can be found some works on the quaternions and octonions for Jabosthal, Pell as well as the Fibonacci and Lucas numbers. For example, in \cite{KeciliogluAkkus}, it is investigated some properties of the Fibonacci and Lucas octonions.
\par
It is natural to wonder whether there exits a connection between the parameters bi-periodicity and quaternions. Actually this was done in \cite{TanYilmazSahin,TanYilmazSahin1}. The authors defined the bi-periodic Fibonacci and Lucas quaternions as
\begin{equation*}
Q_{n}=\sum_{s=0}^{3}q_{n+s}e_{s}\,
\text{ and  }
P_{n}=\sum_{s=0}^{3}l_{n+s}e_{s}\, ,
\end{equation*}
where $n\geq0$. 
\par
In the light of all these above material (depicted as separate paragraphs), the main goal of this paper is to
define \textit{bi-periodic Fibonacci octonions} with a different viewpoint. To do that, by using bi-periodic Fibonacci numbers, we obtain the generating function of these octonions. We also actually investigated some properties of bi-periodic Fibonacci octonions.

\section{The bi-periodic Fibonacci octonions}

\hspace{0.5cm} As we indicated in the previous section, in this part, we introduce the bi-periodic Fibonacci octonions as a new
generalization of Fibonacci octonions by considering the bi-periodic Fibonacci numbers. 
\par
\begin{defi} \label{def 2.1}
The bi-periodic Fibonacci octonions $O_{n}(a,b)$ are defined
by
\begin{equation}
O_{n}(a,b) =\sum_{s=0}^{7}q_{n+s}e_{s},  \label{2.1}
\end{equation}
where $q_{n}$ is the bi-periodic Fibonacci number for $n\in\mathbb{N}=\left\{ 0,1,2,3,\ldots \right\}$.
\end{defi}

\par
From the Equation (\ref{negFib}), that is $q_{-n}=(-1)^{n-1}q_{n}$, the bi-periodic Fibonacci octonions with negative subscripts are defined by
\begin{equation}
O_{-n}(a,b) =\sum_{s=0}^{7}(-1)^{n-s-1}q_{n-s}e_{s},  \label{2.2}
\end{equation}
where $n\in\mathbb{N}=\left\{ 0,1,2,3,\ldots \right\}$.

\par
After all, we present the following first proposition of this paper.
\begin{prop}
For $n\in\mathbb{N}$, we give the results related to bi-periodic Fibonacci octonions:
\begin{itemize}
\item[$i)$]
$O_{n}(a,b)+\overline{O_{n}(a,b)}=2q_{n}e_{0}$,
\item[$ii)$]
$O_{n}^{2}(a,b)+O_{n}(a,b)\overline{O_{n}(a,b)}=2q_{n}e_{0}O_{n}(a,b)$,
\item[$iii)$]
$O_{n}(a,b)\overline{O_{n}(a,b)}=\sum_{s=0}^{7}q_{n+s}^{2}$
\item[$iv)$]
$O_{-n}(a,b)+\overline{O_{-n}(a,b)}=2(-1)^{n-1}q_{n}e_{0}$
\item[$v)$]
$O_{n}(a,b)+O_{-n}(a,b)=l_{n}O_{0}(a,b)$, where $n$ is even.
\end{itemize}
\end{prop}

\begin{proof}
\begin{itemize}
\item[$i)$]From Definition \ref{def 2.1}, we obtain
\begin{eqnarray*}
O_{n}(a,b)+\overline{O_{n}(a,b)}&=&\sum_{s=0}^{7}q_{n+s}e_{s}+q_{n}e_{0}-\sum_{s=1}^{7}q_{n+s}e_{s} \\
&=& 2q_{n}e_{0}\, .
\end{eqnarray*}
\item[$ii)$]By taking into account $i)$, we have
\begin{eqnarray*}
O_{n}^{2}(a,b)&=& O_{n}(a,b)(2q_{n}e_{0}-\overline{O_{n}(a,b)}) \\
&=& 2q_{n}e_{0}O_{n}(a,b)-O_{n}(a,b)\overline{O_{n}(a,b)} \,.
\end{eqnarray*}
\item[$iii)$]The proof is easily seen by using the Equation (\ref{eslenik}).
\item[$iv)$]The proof is similar to $i)$.
\item[$v)$]From the Equations (\ref{2.1}) and (\ref{2.2}), we write for $n$ even
\begin{eqnarray*}
O_{n}(a,b)+O_{-n}(a,b)&=&\sum_{s=0}^{7}q_{n+s}e_{s}+\sum_{s=0}^{7}(-1)^{n-s-1}q_{n-s}e_{s} \\
&=& \sum_{s=1}^{7}(q_{n+s}+(-1)^{s-1}q_{n-s})e_{s}\, .
\end{eqnarray*}
By taking into account the Equations (\ref{rekFib}) and (\ref{bag1}), we obtain
\begin{eqnarray*}
O_{n}(a,b)+O_{-n}(a,b)&=& l_{n}\sum_{s=1}^{7}q_{s}e_{s} \\
&=& l_{n}O_{0}(a,b)\, .
\end{eqnarray*}
\end{itemize}
\end{proof}

\bigskip
\begin{thm}
\label{teo 2.1} For $n\geq0$, the Binet formula for the bi-periodic Fibonacci octonions is
\begin{equation*}
O_{n}(a,b)=\left\{
\begin{array}{c}
\dfrac{1}{(ab)^{\left\lfloor \frac{n}{2}\right\rfloor }}(\dfrac{\alpha^{*}\alpha^n-\beta^{*}\beta^n}{\alpha-\beta}),\ \ \text{if }n\text{ is even} \\
\dfrac{1}{(ab)^{\left\lfloor \frac{n}{2}\right\rfloor }}(\dfrac{\alpha^{**}\alpha^n-\beta^{**}\beta^n}{\alpha-\beta}),\ \ \text{if }n\text{ is odd}\
\end{array}
\right. 
\end{equation*}
where $\alpha,\beta$ are roots of characteristic equation of $\lambda^2-ab\lambda-ab=0$ and \\
$\alpha^{*}=\sum_{t=0}^{7}\dfrac{a^{\xi(t+1)}}{(ab)^{\left\lfloor \frac{t}{2}\right\rfloor }}\alpha^{t}e_{t}$ and $\beta^{*}=\sum_{t=0}^{7}\dfrac{a^{\xi(t+1)}}{(ab)^{\left\lfloor \frac{t}{2}\right\rfloor}}\beta^{t}e_{t}$ , \\
$\alpha^{**}=\sum_{t=0}^{7}\dfrac{a^{\xi(t)}}{(ab)^{\left\lfloor \frac{t+1}{2}\right\rfloor }}\alpha^{t}e_{t}$ and $\beta^{**}=\sum_{t=0}^{7}\dfrac{a^{\xi(t)}}{(ab)^{\left\lfloor \frac{t+1}{2}\right\rfloor }}\beta^{t}e_{t}$ .
\end{thm}

\begin{proof}
It can easily seen by considering the Definition \ref{def 2.1} and the Equation (\ref{binFib}).
\end{proof}
\bigskip
\begin{thm}\label{teo 2.2}
For $n,r\in\mathbb{Z}^{+}$ and $n\geq r$, the following equality is hold:
\begin{equation*}
O_{2n-2r}(a,b)O_{2n+2r}(a,b)-O_{2n}^{2}(a,b)=\frac{\alpha^{*}\beta^{*}((ab)^{2r}-\alpha^{4r})+\beta^{*}\alpha^{*}((ab)^{2r}-\beta^{4r})}{(ab)^{2r}(\alpha-\beta)^{2}} .
\end{equation*}
\end{thm}

\begin{proof}
Let us take $U=O_{2n-2r}(a,b)O_{2n+2r}(a,b)-O_{2n}^{2}(a,b)$. Then, by using Teorem \ref{teo 2.1}, we write 
\begin{eqnarray*}
U &=& \dfrac{1}{(ab)^{\left\lfloor \frac{2n-2r}{2}\right\rfloor }}(\dfrac{\alpha^{*}\alpha^{2n-2r}-\beta^{*}\beta^{2n-2r}}{\alpha-\beta})\dfrac{1}{(ab)^{\left\lfloor \frac{2n+2r}{2}\right\rfloor }}(\dfrac{\alpha^{*}\alpha^{2n+2r}-\beta^{*}\beta^{2n+2r}}{\alpha-\beta})
\\&&-(\dfrac{1}{(ab)^{\left\lfloor \frac{2n}{2}\right\rfloor }}(\dfrac{\alpha^{*}\alpha^{2n}-\beta^{*}\beta^{2n}}{\alpha-\beta}))^{2} \\
&=&\dfrac{1}{(ab)^{2n}(\alpha-\beta)^{2}}[\alpha^{*}\beta^{*}(-ab)^{2n}(1-\alpha^{4r}(-ab)^{-2r})+\beta^{*}\alpha^{*}(-ab)^{2n}(1-\beta^{4r}(-ab)^{-2r})]\\
&=& \dfrac{\alpha^{*}\beta^{*}((ab)^{2r}-\alpha^{4r})
+\beta^{*}\alpha^{*}((ab)^{2r}-\beta^{4r})}{(ab)^{2r}(\alpha-\beta)^{2}}\,.
\end{eqnarray*}
\end{proof}

\bigskip
If we take $r=1$ in Theorem \ref{teo 2.2}, it is obtained in the following corollary.
\begin{cor}
For $n\in\mathbb{Z}^{+}$, we have the Cassini-like identity
\begin{equation*}
O_{2n-2}(a,b)O_{2n+2}(a,b)-O_{2n}^{2}(a,b)=\frac{\alpha^{*}\beta^{*}((ab)^{2}-\alpha^{4})+\beta^{*}\alpha^{*}((ab)^{2}-\beta^{4})}{(ab)^{2}(\alpha-\beta)^{2}} .
\end{equation*}
\end{cor}

\bigskip
\begin{thm}
For $n,r\in\mathbb{Z}^{+}$, $n\geq r$ and $r$ is even, we have
\begin{equation*}
O_{n-r}(a,b)O_{n+r}(a,b)-O_{n}^{2}(a,b)=\left\{
\begin{array}{c}
\frac{\alpha^{*}\beta^{*}((ab)^{r}-\alpha^{2r})+\beta^{*}\alpha^{*}((ab)^{r}-\beta^{2r})}{(ab)^{r}(\alpha-\beta)^{2}},\ \ \ \ \ \ \ \ \text{if }n\text{ is even} \\
-\frac{\alpha^{**}\beta^{**}((ab)^{r}-\alpha^{2r})+\beta^{**}\alpha^{**}((ab)^{r}-\beta^{2r})}{(ab)^{r-1}(\alpha-\beta)^{2}},\ \ \text{if }n\text{ is odd}
\end{array}
\right..
\end{equation*}
\end{thm}

\begin{proof}
The proof is done similarly to the proof of Theorem \ref{teo 2.2}.
\end{proof}
\bigskip
\par
Now, we present the generating function of the bi-periodic Fibonacci octonions in the following theorem.
\begin{thm}
For the bi-periodic Fibonacci octonions, we have the generating function
\begin{equation*}
\sum\limits_{n=0}^{\infty}{O}_{n}(a,b)x^{n}=\dfrac{O_{0}(a,b)+x(O_{1}(a,b)-bO_{0}(a,b))+(a-b)R(x)}{1-bx-x^{2}} \,,
\end{equation*}
where \\ 
$R(x)=(xe_{0}+e_{1}+\dfrac{1}{x}e_{2}+\dfrac{1}{x^2}e_{3}+\dfrac{1}{x^3}e_{4}+\dfrac{1}{x^4}e_{5}++\dfrac{1}{x^5}e_{6}+\dfrac{1}{x^6}e_{7})f(x)-(xe_{1}+e_{2}+(\dfrac{1}{x}+(ab+1)x)e_{3}+(\dfrac{1}{x^2}+ab+1)e_{4}+(\dfrac{1}{x^3}+(ab+1)\dfrac{1}{x}+(a^2b^2+3ab+1)x)e_{5}+(\dfrac{1}{x^4}+(ab+1)\dfrac{1}{x^2}+(a^2b^2+3ab+1))e_{6}+(\dfrac{1}{x^5}+(ab+1)\dfrac{1}{x^3}+(a^2b^2+3ab+1)\dfrac{1}{x}+(a^3b^3+5a^2b^2+6ab+1)x)e_{7})$\\ and
$f(x)=\dfrac{x-x^3}{1-(ab+2)x^2+x^4}$.
\end{thm}

\begin{proof}
Assume that $G(x,a,b)$ is the generating function for the bi-periodic Fibonacci octonions. Then we have 
\begin{equation}
G(x,a,b)=O_{0}(a,b)+O_{1}(a,b)x+O_{2}(a,b)x^{2}+\cdots+O_{n}(a,b)x^{n}+\cdots.
\label{*}
\end{equation}
If it is multiplying equation (\ref{*}) with $bx$ and $x^{2}$,
respectively, then we have
\begin{equation}
bxG(x,a,b)=bO_{0}(a,b)x+bO_{1}(a,b)x^{2}+bO_{2}(a,b)x^{3}+\cdots+bO_{n}(a,b)x^{n+1}+\cdots.  \label{*0}
\end{equation}
\begin{equation}
x^{2}G(x,a,b)=O_{0}(a,b)x^{2}+O_{1}(a,b)x^{3}+O_{2}(a,b)x^{4}+\cdots+O_{n}(a,b)x^{n+2}+\cdots.
\label{***}
\end{equation}
By considering the Equations (\ref{rekFib}), (\ref{*}), (\ref{*0}), (\ref{***}) and Definition \ref{def 2.1},  it is obtained the equation
\begin{eqnarray*}
(1-bx-x^{2})G(x)&=& O_{0}(a,b)+x(O_{1}(a,b)-bO_{0}(a,b))\\&&
+\sum_{n=2}^{\infty}(O_{n}(a,b)-bO_{n-1}(a,b)-O_{n-2}(a,b))x^{n}\\&=& O_{0}(a,b)+x(O_{1}(a,b)-bO_{0}(a,b))\\&&
+\sum_{n=2}^{\infty}\sum_{s=0}^{7}(q_{n+s}-bq_{n+s-1}-q_{n+s-2})e_{s}x^{n}\,.
\end{eqnarray*}
By considering again the Equations (\ref{rekFib}) and (\ref {ureFib}), we get
\begin{eqnarray*}
(1-bx-x^{2})G(x)&=& O_{0}(a,b)+x(O_{1}(a,b)-bO_{0}(a,b))+(a-b)\sum_{n=1}^{\infty}q_{2n-1}x^{2n}e_{0}\\&& +(a-b)\sum_{n=2}^{\infty}q_{2n-1}x^{2n-1}e_{1}+(a-b)\sum_{n=2}^{\infty}q_{2n-1}x^{2n-2}e_{2}\\&& +(a-b)\sum_{n=3}^{\infty}q_{2n-1}x^{2n-3}e_{3}+(a-b)\sum_{n=3}^{\infty}q_{2n-1}x^{2n-4}e_{4}\\&&+(a-b)\sum_{n=4}^{\infty}q_{2n-1}x^{2n-5}e_{5}+(a-b)\sum_{n=4}^{\infty}q_{2n-1}x^{2n-6}e_{6}\\&&+(a-b)\sum_{n=5}^{\infty}q_{2n-1}x^{2n-7}e_{7}\,.
\end{eqnarray*}
Consequently, we have
\begin{eqnarray*}
(1-bx-x^{2})G(x)&=& O_{0}(a,b)+x(O_{1}(a,b)-bO_{0}(a,b))+(a-b)xf(x)e_{0}\\&&
+(a-b)(f(x)-q_{1}x)e_{1}+(a-b)x^{-1}(f(x)-q_{1}x)e_{2}\\&&
+(a-b)x^{-2}(f(x)-q_{1}x-q_{3}x^3)e_{3}\\&&+(a-b)x^{-3}(f(x)-q_{1}x-q_{3}x^3)e_{4}\\&&
+(a-b)x^{-4}(f(x)-q_{1}x-q_{3}x^3-q_{5}x^5)e_{5}\\&&+(a-b)x^{-5}(f(x)-q_{1}x-q_{3}x^3-q_{5}x^5)e_{6}\\&&+(a-b)x^{-6}(f(x)-q_{1}x-q_{3}x^3-q_{5}x^5-q_{7}x^7)e_{7}
\\&=& O_{0}(a,b)+x(O_{1}(a,b)-bO_{0}(a,b))+(a-b)R(x)
\end{eqnarray*}
as required.
\end{proof}

\bigskip
For bi-periodic Fibonacci octonions, we give the summations according to
specified rules.

\begin{thm}
For $n\in \mathbb{Z}^{+}$, there exist
\begin{itemize}
\item[$i)$]
$\sum_{r=0}^{n-1}O_{r}(a,b)=\dfrac{O_{n+1}(a,b)+O_{n}(a,b)-O_{n-1}(a,b)-O_{n-2}(a,b)}{ab}\\
+\dfrac{\alpha^{*}\beta-\beta^{*}\alpha-\alpha^{**}ab+\beta^{**}ab}{ab(\alpha-\beta)}$ \,, \label{R}
\item[$ii)$]
$\sum_{r=0}^{n-1}O_{2r}(a,b)=\dfrac{O_{2n}(a,b)-O_{2n-2}(a,b)}{ab}+\dfrac{\alpha^{*}\beta-\beta^{*}\alpha}{ab(\alpha-\beta)}$\,,
\item[$iii)$]
$\sum_{r=0}^{n-1}O_{2r+1}(a,b)=\dfrac{O_{2n+1}(a,b)-O_{2n-1}(a,b)}{ab}-\dfrac{\alpha^{**}-\beta^{**}}{\alpha-\beta}$ \,.
\end{itemize}
\end{thm}

\begin{proof}
We will, now, just prove $i)$, since the proofs of  $ii)$ and  $iii)$ can be done quite similarly with it. The main point of the proof will be touched just the result Theorem $\ref{teo 2.1}$(in other words the Binet formula of these octonions). We must note that the proof should be investigated for both cases of $n$.

\textbf{\underline{$\mathit{n}$ is odd }:} In this case, we get
\begin{eqnarray*}
\sum\limits_{r=0}^{n-1}O_{r}(a,b) &=&\sum\limits_{r=0}^{\frac{n-1}{2}}O_{2r}(a,b)+\sum\limits_{r=0}^{\frac{n-3}{2}}O_{2r+1}(a,b)\\
&=&\sum\limits_{r=0}^{\frac{n-1}{2}}\dfrac{1}{(ab)^r}(\dfrac{\alpha^{*}\alpha^{2r}-\beta^{*}\beta^{2r}}{\alpha-\beta})+\sum\limits_{r=0}^{\frac{n-3}{2}}\dfrac{1}{(ab)^r}(\dfrac{\alpha^{**}\alpha^{2r+1}-\beta^{**}\beta^{2r+1}}{\alpha-\beta}) \\
&=&\dfrac{\alpha^{*}}{\alpha-\beta}\sum\limits_{r=0}^{\frac{n-1}{2}}(\dfrac{\alpha^{2}}{ab})^{r}-\dfrac{\beta^{*}}{\alpha-\beta}\sum\limits_{r=0}^{\frac{n-1}{2}}(\dfrac{\beta^{2}}{ab})^{r}+\dfrac{\alpha^{**}}{\alpha-\beta}\sum\limits_{r=0}^{\frac{n-3}{2}}\alpha(\dfrac{\alpha^{2}}{ab})^{r}\\&&-\dfrac{\beta^{**}}{\alpha-\beta}\sum\limits_{r=0}^{\frac{n-3}{2}}\beta(\dfrac{\beta^{2}}{ab})^{r} 
\end{eqnarray*}
\begin{eqnarray*}
\sum\limits_{r=0}^{n-1}O_{r}(a,b) &=&\dfrac{\alpha^{*}}{\alpha-\beta}\dfrac{(\dfrac{\alpha^{2}}{ab})^{\frac{n-1}{2}+1}-1}{\dfrac{\alpha^{2}}{ab}-1}-\dfrac{\beta^{*}}{\alpha-\beta}\dfrac{(\dfrac{\beta^{2}}{ab})^{\frac{n-1}{2}+1}-1}{\dfrac{\beta^{2}}{ab}-1} \\&&+\dfrac{\alpha^{**}}{\alpha-\beta}\dfrac{(\dfrac{\alpha^{2}}{ab})^{\frac{n-3}{2}+1}\alpha-\alpha}{\dfrac{\alpha^{2}}{ab}-1}-\dfrac{\beta^{**}}{\alpha-\beta}\dfrac{(\dfrac{\beta^{2}}{ab})^{\frac{n-3}{2}+1}\beta-\beta}{\dfrac{\beta^{2}}{ab}-1}\,.
\end{eqnarray*}
In here, simplifying the last equality, we imply the equality in $i)$ as
required.

\textbf{\underline{$\mathit{n}$ is even }: }From Theorem \ref{teo 2.1}, we know
\begin{eqnarray*}
\sum\limits_{r=0}^{n-1}O_{r}(a,b) &=&\sum\limits_{r=0}^{\frac{n-2}{2}}O_{2r}(a,b)+\sum\limits_{r=0}^{\frac{n-2}{2}}O_{2r+1}(a,b)\\
&=&\sum\limits_{r=0}^{\frac{n-2}{2}}\dfrac{1}{(ab)^r}(\dfrac{\alpha^{*}\alpha^{2r}-\beta^{*}\beta^{2r}}{\alpha-\beta})+\sum\limits_{r=0}^{\frac{n-2}{2}}\dfrac{1}{(ab)^r}(\dfrac{\alpha^{**}\alpha^{2r+1}-\beta^{**}\beta^{2r+1}}{\alpha-\beta}) \\
&=&\dfrac{\alpha^{*}}{\alpha-\beta}\dfrac{(\dfrac{\alpha^{2}}{ab})^{\frac{n-2}{2}+1}-1}{\dfrac{\alpha^{2}}{ab}-1}-\dfrac{\beta^{*}}{\alpha-\beta}\dfrac{(\dfrac{\beta^{2}}{ab})^{\frac{n-2}{2}+1}-1}{\dfrac{\beta^{2}}{ab}-1}\\&&+\dfrac{\alpha^{**}}{\alpha-\beta}\dfrac{(\dfrac{\alpha^{2}}{ab})^{\frac{n-2}{2}+1}\alpha-\alpha}{\dfrac{\alpha^{2}}{ab}-1}-\dfrac{\beta^{**}}{\alpha-\beta}\dfrac{(\dfrac{\beta^{2}}{ab})^{\frac{n-2}{2}+1}\beta-\beta}{\dfrac{\beta^{2}}{ab}-1}\,.
\end{eqnarray*}
As a result, by arranging the last equality, we have
\begin{equation*}
\sum_{r=0}^{n-1}O_{r}(a,b)=\dfrac{O_{n+1}(a,b)+O_{n}(a,b)-O_{n-1}(a,b)-O_{n-2}(a,b)}{ab}\\
+\dfrac{\alpha^{*}\beta-\beta^{*}\alpha-\alpha^{**}ab+\beta^{**}ab}{ab(\alpha-\beta)}\,.
\end{equation*}
\end{proof}

\bigskip
\begin{cor}
In the all results of Section 2, to reveal the importance of this subject, we can express
certain and immediate relationships as follows:
\begin{itemize}
\item If we replace $a=b=1$ in $q_{n}$, we obtain the same result in \cite{KeciliogluAkkus} for Fibonacci octonions.
\item If we replace $a=b=2$ in $q_{n}$, we obtain the same result in \cite{SzynalWloch1} for Pell octonions.
\item If we replace $a=b=k$ in $q_{n}$, we get the $k$-Fibonacci octonions.
\end{itemize}
\end{cor}

\bigskip
\begin{con}
In this paper, we define the octonions for bi-periodic Fibonacci sequences and present some properties of these octonions. By the results in Sections 2 of this paper, we have a great opportunity to compare and obtain some new properties over these octonions. This is the main aim of this paper. Thus, we extend some recent result in the literature.

In the future studies on the octonions for number sequences, we except that the following topics will bring a new insight.

\begin{itemize}
\item It would be interesting to study the different summations and properties for bi-periodic Fibonacci octonions,
\item Also, it would be interesting to investigate the octonions for bi-periodic generalized Fibonacci sequences.
\end{itemize}
\end{con}

\begin{ack}
This research is supported by TUBITAK and Selcuk University Scientific
Research Project Coordinatorship (BAP).
\end{ack}

\end{document}